%% file: MultDMD.tex
\PassOptionsToPackage{table,xcdraw}{xcolor}
\documentclass[onefignum,onetabnum,final]{siamonline220329}

\input{ex_shared}

\newcommand{\x}{\mathbf{x}}
\newcommand{\y}{\mathbf{y}}

\newcommand{\R}{\mathbb{R}}
\newcommand{\C}{\mathbb{C}}
\newcommand{\K}{\mathbb{K}}

\newcommand{\Cv}{\mathbf{C}}

\newcommand{\Fv}{\mathbf{F}}
\newcommand{\Gv}{\mathbf{G}}

\newcommand{\Kv}{\mathbf{K}}

\newcommand{\Wv}{\mathbf{W}}

\newcommand{\Psiv}{\mathbf{\Psi}}
\newcommand{\muv}{{\bm\mu}}

\newcommand{\F}{\textup{F}}
\newcommand{\gv}{\mathbf{g}}
\newcommand{\hv}{\mathbf{h}}
\newcommand{\fv}{\mathbf{f}}

\newcommand*{\dd}{{\,\mathrm{d}}}
\newcommand*{\pp}{{[{-}\pi,\pi]_{\mathrm{per}}}}

\newcommand*{\spec}{{\mathrm{Sp}}}

\newcommand{\kdmd}{\mathbf{K}_{\text{EDMD}}}
\newcommand{\kmpd}{\mathbf{K}_{\text{mpEDMD}}}
\newcommand{\kmult}{\mathbf{K}_{\text{MultDMD}}}

\renewcommand{\K}{\mathcal{K}}
\usepackage{bm}
\DeclareMathOperator*{\diag}{diag}

\begin{document}

\maketitle

\begin{abstract}
	Koopman operators are infinite-dimensional operators that linearize nonlinear dynamical systems, facilitating the study of their spectral properties and enabling the prediction of the time evolution of observable quantities. Recent methods have aimed to approximate Koopman operators while preserving key structures. However, approximating Koopman operators typically requires a dictionary of observables to capture the system's behavior in a finite-dimensional subspace. The selection of these functions is often heuristic, may result in the loss of spectral information, and can severely complicate structure preservation. This paper introduces Multiplicative Dynamic Mode Decomposition (MultDMD), which enforces the multiplicative structure inherent in the Koopman operator within its finite-dimensional approximation. Leveraging this multiplicative property, we guide the selection of observables and define a constrained optimization problem for the matrix approximation, which can be efficiently solved. MultDMD presents a structured approach to finite-dimensional approximations and can more accurately reflect the spectral properties of the Koopman operator. We elaborate on the theoretical framework of MultDMD, detailing its formulation, optimization strategy, and convergence properties. The efficacy of MultDMD is demonstrated through several examples, including the nonlinear pendulum, the Lorenz system, and fluid dynamics data, where we demonstrate its remarkable robustness to noise.
\end{abstract}



\section{Introduction}
We consider discrete-time dynamical systems of the form:
\begin{equation} \label{eq_dyn_system}
	\x_{n+1} = \Fv(\x_n), \quad n\geq 0,
\end{equation}
where $\Omega\subset\R^d$ is the state space, $\Fv:\Omega\to\Omega$ is an unknown nonlinear map, and $\x_0\in \Omega$. We aim to study this system through a global linearization in infinite dimensions by analyzing its Koopman operator~\cite{koopman1931hamiltonian, koopman1932dynamical}. For a positive measure $\omega$, the Koopman operator $\K:L^2(\Omega,\omega)\to L^2(\Omega,\omega)$ acts on square-integrable functions $g\in L^2(\Omega,\omega)$ as
\begin{equation} \label{eq_koopman_op}
	[\mathcal{K} g](\x)=(g\circ \Fv)(\x)=g(\Fv(\x)), \quad \x\in\Omega.
\end{equation}
The functions $g$ are called observables and measure the state of the system since
$[\K g](\x_n) = g(\x_{n+1})$ for $n\geq 0$.
We are interested in the data-driven recovery of spectral properties of the Koopman operator, such as its eigenfunctions and eigenvalues. The spectral information of $\K$ contains substantial information about the dynamical system, and this area has attracted a surge of interest over the last decade~\cite{mezic2013analysis,budivsic2012applied,brunton2021modern,colbrook2023multiverse}.

\begin{figure}[t!]
	\centering
	\begin{overpic}[width=\textwidth]{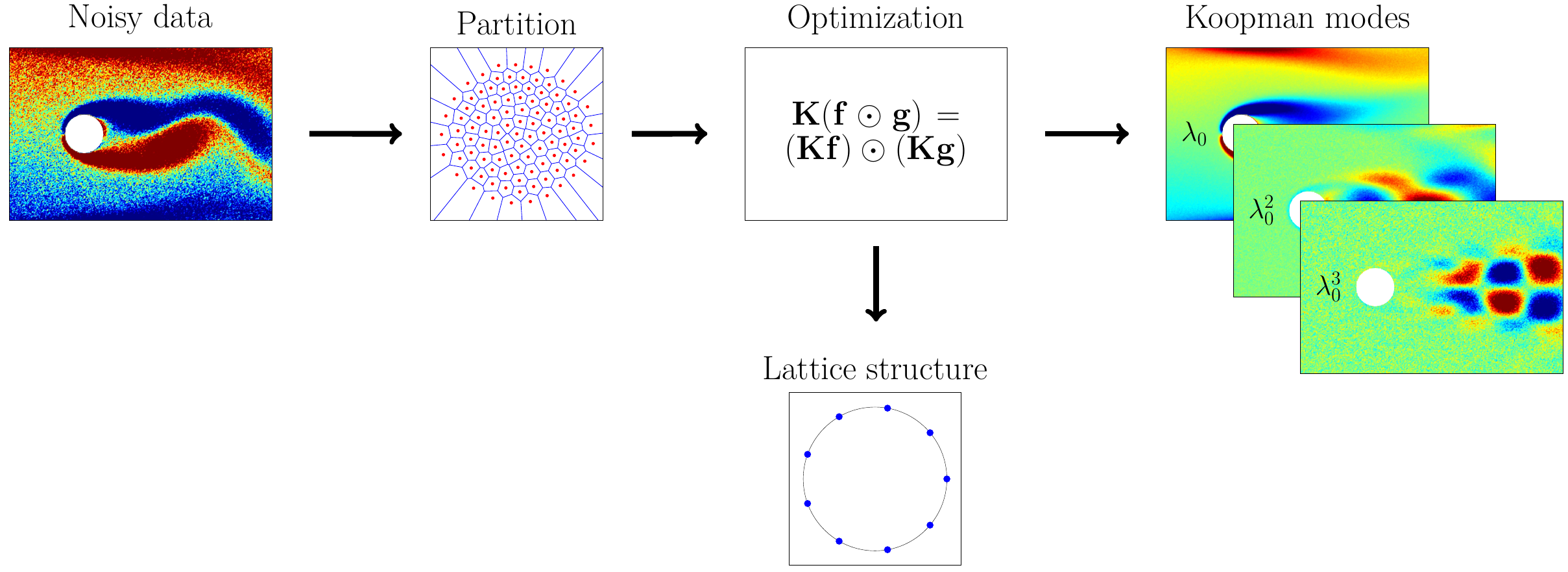}
	\end{overpic}
	\caption{Schematic of the Multiplicative Dynamic Mode Decomposition (MultDMD) method.}
	\label{fig_schema}
\end{figure}

We assume that we have access to $M\geq 1$ snapshot pairs of the system:
\begin{equation}
	\label{eq_snapshot_data}
	\{(\x^{(m)},\y^{(m)}=\Fv(\x^{(m)}))\}_{m=1}^M,
\end{equation}
where the data could come from a single trajectory (e.g., if the dynamical system is ergodic) or multiple trajectories. A popular method for data-driven approximation of the spectral content of $\K$ is \textit{Dynamic Mode Decomposition} (DMD), first introduced in the fluids community~\cite{schmid2009dynamic,schmid2010dynamic} and later connected to Koopman operators~\cite{rowley2009spectral}. DMD forms a Galerkin approximation of the action of $\K$ on a finite-dimensional space. Extended DMD (EDMD) \cite{williams2015data} makes this connection explicit by considering a dictionary (basis) of (possibly) nonlinear observables. There is much interest in establishing convergence results for DMD-type methods to their infinite-dimensional Koopman counterparts. For example, \textit{Residual Dynamic Mode Decomposition} (ResDMD)~\cite{colbrook2021rigorous,colbrook2023residualJFM} explicitly measures, controls, and minimizes the projection errors associated with EDMD, ensuring convergence to the spectral properties of $\mathcal{K}$. We refer to the survey \cite{colbrook2023multiverse} for more information on DMD methods and their convergence properties.

Recent attention has been given to structure-preserving DMD methods. For example, Baddoo et al. provided a framework, \textit{physics-informed DMD} (piDMD), for imposing constraints within DMD and a \textit{linear} choice of dictionary \cite{baddoo2023physics}. Another example is \textit{measure-preserving EDMD} (mpEDMD) \cite{colbrook2023mpedmd}, which ensures that the EDMD approximation preserves the measure $\omega$ and has the benefit of being dictionary agnostic. When designing structure-preserving methods, there is an important distinction between DMD and EDMD:
\begin{itemize}
	\item DMD constructs a matrix that acts on the state space, effectively serving as a global linearization of the dynamical system;
	\item EDMD forms a finite-dimensional approximation of $\K$ in coefficient space.
\end{itemize}
This difference can make it challenging to enforce that EDMD methods preserve structures, as the choice of dictionary often determines the complexity of the optimization problem.\footnote{Fortunately, this issue is mitigated with mpEDMD.}

A useful property of $\K$, and the focus of this paper, is its multiplicative structure: if $f,g\in L^2(\Omega,\omega)$ such that $fg\in L^2(\Omega,\omega)$, then
\begin{equation} \label{eq_koopman_mult}
	[\K(fg)](\x) = [\K(f)](\x)[\K(g)](\x),\quad \x\in\Omega.
\end{equation}
For example, suppose that $\lambda_1$ and $\lambda_2$ are eigenvalues with eigenfunctions $\phi_1$ and $\phi_2$, respectively. Then, assuming all relevant observables lie in $L^2(\Omega,\omega)$, we have
\[\K(\phi_1^n\phi_2^m)=\lambda_1^{n}\lambda_2^{m}\phi_1^n\phi_2^m,\quad n,m\in\mathbb{N}.\]
This property induces a lattice or group structure on the spectrum of the Koopman operator. Throughout this paper, we assume that $\K$ is unitary, which is equivalent to the dynamical system being measure-preserving and invertible (up to $\omega$-null sets)~\cite[Chapt.~7]{eisner2015operator}. Ridge showed that if $\K$ is unitary, then both $\spec(\K)$ and the set of eigenvalues of $\K$ are unions of subgroups of $\mathbb{T}=\{z\in\mathbb{C}:|z|=1\}$~\cite{ridge1973spectrum}. In other words, \cref{eq_koopman_mult} implies a circular symmetry of spectral properties of $\K$.

We introduce a new method called \textit{Multiplicative Dynamic Mode Decomposition} (MultDMD), that aims to preserve the multiplicative structure of the Koopman operator in \cref{eq_koopman_mult} on its finite-dimensional approximation. This leads to a natural choice of basis functions and a specific constrained optimization problem to find the matrix approximation of the action of the Koopman operator on the resulting finite-dimensional subspace. The matrix can then be used to compute eigenvalues and eigenfunctions of the Koopman operator and extract information about the system's dynamics. MultDMD enforces \cref{eq_koopman_mult} irrespective of the number of basis functions or data snapshots. Additionally, the structure of the optimization problem (constrained least-squares) can be exploited to develop an efficient algorithm. A schematic of the MultDMD method is shown in~\cref{fig_schema}.

The paper is organized as follows. We begin by introducing EDMD and related works on structure-preserving DMD-type methods in \cref{sec_preliminary}. Then, in \cref{sec_multiplicative}, we present a suitable dictionary of basis functions, along with an efficient constrained optimization algorithm for preserving the multiplicative structure of the Koopman operator on its finite-dimensional approximant. Simple conditions for convergence are given in \cref{sec_convergence}. We analyze the performance of MultDMD in \cref{sec_examples_MultDMD} through several numerical examples. These demonstrate its ability for data-driven discovery of coherent features, including singular generalized eigenfunctions of the pendulum and Lorenz systems, and its efficiency and robustness, even in the presence of severe noise. Finally, we summarize our main results and conclude in \cref{sec_conclusions}.  General purpose code and all the examples of this paper can be found at:
\url{https://github.com/NBoulle/MultDMD}.

\section{Background and related works}\label{sec_preliminary}

\subsection{Extended dynamic mode decomposition} \label{sec_edmd}

EDMD constructs a finite-dimensional approximation of $\mathcal{K}$ from the snapshot data in~\cref{eq_snapshot_data}. Given a dictionary $\{\psi_1,\ldots,\psi_{N}\}\subset L^2(\Omega,\omega)$, we approximate the action of $\mathcal{K}$ on the subspace $V_N=\mathrm{span}\{\psi_1,\ldots,\psi_{N}\}$ by a matrix $\kdmd\in\mathbb{C}^{N\times N}$ as
\[
	{[\mathcal{K}\psi_j](\x) = \psi_j(\Fv(\x)) \approx \sum_{i=1}^{N} (\kdmd)_{ij} \psi_i(\x)},\quad1\leq j\leq {N}.
\]
For notational convenience, define $\Psi(\x)=\begin{bmatrix}\psi_1(\x) & \cdots& \psi_{{N}}(\x) \end{bmatrix}\in\mathbb{C}^{1\times {N}}$ so that any $g\in V_{N}$ can be written as $g(\x)=\sum_{j=1}^{N}\psi_j(\x)g_j=\Psi(\x)\,\gv$ for $\gv\in\mathbb{C}^{N}$. EDMD considers
\begin{equation} \label{eq_ContinuousLeastSquaresProblem}
	\kdmd\approx\argmin_{\Kv\in\mathbb{C}^{N\times N}} \int_\Omega \left\|\Psi(\Fv(\x))\Cv^{-1} - \Psi(\x)\Kv\Cv^{-1}\right\|^2_{\ell^2}\dd\omega(\x),
\end{equation}
where $\|\cdot\|_{\ell^2}$ denotes the standard Euclidean norm of a vector, and $\Cv$ is an invertible positive self-adjoint matrix that controls the size of $g=\Psi\gv$. The integral in~\eqref{eq_ContinuousLeastSquaresProblem} cannot be evaluated directly, so it is approximated using a quadrature rule with nodes $\{\x^{(m)}\}_{m=1}^{M}$ and weights $\{w_m\}_{m=1}^{M}$.  For notational convenience, we introduce the matrices $\Wv=\mathrm{diag}(w_1,\ldots,w_{M})$ and
\begin{equation}
	\begin{split}
		\Psiv_X=\begin{pmatrix}
			        \Psi(\x^{(1)}) \\
			        \vdots         \\
			        \Psi(\x^{(M)})
		        \end{pmatrix}\in\mathbb{C}^{M\times N},\quad
		\Psiv_Y=\begin{pmatrix}
			        \Psi(\y^{(1)}) \\
			        \vdots         \\
			        \Psi(\y^{(M)})
		        \end{pmatrix}\in\mathbb{C}^{M\times N}.
		\label{psidef}
	\end{split}
\end{equation}
The discretized version of~\eqref{eq_ContinuousLeastSquaresProblem} yields the following weighted least-squares problem:
\begin{equation} \label{EDMD_opt_prob2}
	\kdmd\coloneqq \argmin_{\Kv\in\mathbb{C}^{N\times N}}\left\|\Wv^{1/2}\Psiv_Y\Cv^{-1}-\Wv^{1/2}\Psiv_X \Kv\Cv^{-1}\right\|_\F^2,
\end{equation}
whose solution is given by
$
	\kdmd= (\Wv^{1/2}\Psiv_X)^\dagger \Wv^{1/2}\Psiv_Y=(\Psiv_X^*\Wv\Psiv_X)^\dagger\Psiv_X^*\Wv\Psiv_Y,
$
where `$\dagger$' denotes the Moore--Penrose pseudoinverse. If the quadrature converges\footnote{See \cite{colbrook2023multiverse} for a discussion of when this holds.} then
\begin{equation}
	\label{quad_convergence}
	\lim_{M\rightarrow\infty}[\Psiv_X^*\Wv\Psiv_X]_{jk} = \langle \psi_k,\psi_j \rangle\quad \text{ and }\quad \lim_{M\rightarrow\infty}[\Psiv_X^*\Wv\Psiv_Y]_{jk} = \langle \mathcal{K}\psi_k,\psi_j \rangle,
\end{equation}
where $\langle \cdot,\cdot \rangle$ is the inner product associated with $L^2(\Omega,\omega)$. Hence, in the large data limit, $\kdmd=(\Psiv_X^*\Wv\Psiv_X)^\dagger\Psiv_X^*\Wv\Psiv_Y$ approaches a matrix representation of $\mathcal{P}_{V_{N}}\mathcal{K}\mathcal{P}_{V_{N}}^*$, where $\mathcal{P}_{V_{N}}$ denotes the orthogonal projection onto $V_{N}$. In essence, EDMD is a Galerkin method. The EDMD eigenvalues thus approach the spectrum of $\mathcal{P}_{V_{N}}\mathcal{K}\mathcal{P}_{V_{N}}^*$, and EDMD is an example of the so-called finite section method~\cite{bottcher1983finite}. Since the finite section method can suffer from spectral pollution (spurious modes), spectral pollution is a concern for EDMD~\cite{williams2015data}. This and other issues related to EDMD are resolved using Residual DMD \cite{colbrook2021rigorous,colbrook2023residualJFM}.

\paragraph{Adding constraints} With additional information about the dynamical system, we can often enforce physically motivated constraints in the minimization problem~\eqref{EDMD_opt_prob2}. This is explored extensively in~\cite{baddoo2023physics} for various types of systems in the context of DMD. Before discussing connections with the broader literature, we discuss the two methods most closely related to the current paper: measure-preserving EDMD and periodic approximations.

\subsection{Measure-preserving EDMD} \label{sec_mpEDMD}
Measure-preserving EDMD (mpEDMD) \cite{colbrook2023mpedmd} enforces that the EDMD approximation is measure-preserving. We can approximate the inner product $\langle \cdot,\cdot\rangle$  via the inner product induced by the matrix $\Gv=\Psiv_X^*\Wv\Psiv_X$ as
\begin{equation}
	\label{inner_product_G}
	\hv^*\Gv\gv=\sum_{j,k=1}^N  \overline{h_j}g_kG_{j,k}\approx \sum_{j,k=1}^N  \overline{h_j}g_k\langle \psi_k,\psi_j \rangle=\langle \Psi \gv,\Psi \hv \rangle.
\end{equation}
If the convergence in~\eqref{quad_convergence} holds, then $\lim_{M\rightarrow\infty}\hv^*\Gv\gv=\langle \Psi \gv,\Psi \hv \rangle$. Hence, $\|\Psi \gv\|^2\approx \gv^*\Gv\gv$ and $\|\Psi \Kv\,\gv\|^2\approx \gv^*\Kv^*\Gv\Kv\gv$. Since $\mathcal{K}$ is an isometry, $\|\K g\|^2=\|g\|^2$ and we enforce $\Kv^*\Gv\Kv=\Gv$. We set $\Cv=\Gv^{1/2}$ and replace~\eqref{EDMD_opt_prob2} by
\begin{equation}
	\label{mpEDMD_opt_prob}
	\kmpd\coloneqq \underset{\substack{\Kv^*\Gv\Kv=\Gv\\\Kv\in\mathbb{C}^{N\times N}}}{\argmin}\left\|\Wv^{1/2}\Psiv_Y\Gv^{-1/2}-\Wv^{1/2}\Psiv_X \Kv\Gv^{-1/2}\right\|_\F^2,
\end{equation}
which can be solved using an SVD. Therefore, we enforce that the Galerkin approximation is an isometry with respect to the learned inner product induced by $\Gv$. The eigendecomposition of $\kmpd$ converges to the spectral quantities of Koopman operators for general measure-preserving dynamical systems (see the discussion about weak convergence later in \cref{sec:weak_conv}). This result fundamentally depends on the additional constraint in~\eqref{mpEDMD_opt_prob}. Other benefits include increased robustness to noise and conservation of key properties (e.g., the energy if this is what $\omega$ measures).

\subsection{Periodic approximations} \label{sec_periodic}
While mpEDMD allows for a generic dictionary, an alternative for constraining EDMD involves selecting a dictionary that simplifies integrating constraints into the optimization problem. As an example, \cite{govindarajan2019approximation,govindarajan2021approximation} offer a method akin to the Ulam approximation \cite{ulam1960collection,li1976finite} for the Perron--Frobenius operator, constructing \textit{periodic approximations} of Koopman operators under specific conditions: compact $\Omega$, $\omega$ being absolutely continuous with respect to Lebesgue measure, and the system being both measure-preserving and invertible. This approach has been adapted into an algorithm for tori systems and extended to continuous-time systems \cite{govindarajan2021approximation}, relying on state-space partitioning to approximate the Koopman operator's dynamics through a permutation. This method yields measures that converge weakly to the spectral measures of the Koopman operator. Furthermore, periodic approximations are positive operators and uphold the multiplicative structure of the Koopman operator. However, generalizing these results to systems beyond those with invariant Lebesgue absolutely continuous measures, such as chaotic attractors, and developing efficient high-dimensional schemes remain open challenges.

\subsection{Other methods}
The Koopman generator is skew-adjoint for continuous-time, invertible, measure-preserving systems, and various methods have been developed to approximate these generators. The authors of \cite{das2021reproducing} introduce an approach using a one-parameter family of reproducing kernels to approximate Koopman eigenvalues and eigenfunctions within the $\epsilon$-pseudospectrum, where $\epsilon$ is determined by a Dirichlet energy functional. Recently, \cite{valva2023consistent} combines this compactification strategy with the integral representation of the resolvent \cite{susuki2021koopman} to create a compact operator that approximates the resolvent of a skew-adjoint operator. This leads to skew-adjoint unbounded operators with compact resolvents whose spectral measures weakly converge to those of the Koopman generator. Other structure-preserving DMD methods include naturally structured DMD~\cite{huang2018data}, DMD for dynamical systems with symmetries characterized by a finite group~\cite{salova2019koopman}, Lagrangian DMD~\cite{lu2020lagrangian}, Port-Hamiltonian DMD~\cite{morandin2023port}, symmetric DMD~\cite{cohen2020mode}, constrained DMD~\cite{krake2022constrained}. Recently in~\cite{boulle2024convergence}, the authors proved that Hermitian DMD~\cite{baddoo2023physics,drmavc2024hermitian} converges for Hermitian Koopman operators, as well as skew-Hermitian operators using an appropriate multiplication by $i$.

\section{Multiplicative dynamic mode decomposition (MultDMD)} \label{sec_multiplicative}

We aim to construct an approximation of the Koopman operator, $\kmult$, on the finite-dimensional subspace $V_N=\mathrm{span}\{\psi_1,\ldots,\psi_N\}$ that maintains the multiplicative property \cref{eq_koopman_mult}. This involves assuming $V_N$ is closed under multiplication and ensuring $\kmult$ upholds this property.

\subsection{Optimization problem and algorithm} \label{sec_optimization}

The multiplicative property \cref{eq_koopman_mult} leads to a natural choice of basis functions as piecewise constant functions with disjoint support, \emph{i.e.}, a discontinuous Galerkin approximation of degree zero of the Koopman operator. We set
$$
	\psi_j(\x) = \chi_{S_j}(\x)=\begin{cases} 1,\quad\text{if }\x\in S_j, \\
		0,\quad \text{otherwise,}
	\end{cases}
$$
where $S_1,\ldots,S_N\subset\mathbb{R}^d$ are disjoint, $\Omega\subset \cup_{j=1}^NS_j$ and $0<\omega(S_j\cap\Omega)<\infty$. We can view each function $\psi_j$ as a member of $L^2(\Omega,\omega)$, even if $S_j\backslash \Omega\neq\emptyset$. This choice ensures that
$$
	f(\x) = \Psi(\x) \fv,\quad g(\x) = \Psi(\x)\gv, \quad f(\x)g(\x) = \Psi(\x)(\fv \odot\gv), \quad f,g\in V_N,\quad \x\in\Omega.
$$
Here, $\odot$ denotes the Hadamard product $(\fv \odot \gv)_j = f_j g_j$ for $\fv,\gv\in\mathbb{C}^N$. To preserve the multiplicative structure at the discrete level, we are therefore led to the constraint
$$
	\Psi(\x)(\Kv \fv\odot \gv)\approx [\K(f g)](\x)  =[\K f](\x)[\K g](\x)   \approx  \Psi(\x)[( \Kv\fv) \odot (\Kv \gv)].
$$
Hence, we adapt EDMD to approximate $\K$ by a matrix $\kmult \in \C^{N\times N}$ satisfying
\begin{equation} \label{eq_mult_matrix}
	\kmult [\fv\odot \gv]=(\kmult \fv)\odot (\kmult \gv),\quad \fv,\gv\in \C^N.
\end{equation}
Since the supports of the basis functions are disjoint, the Gram matrix $\Gv$ is diagonal with
\[
	\Gv=\diag(G_1,\ldots, G_N),\quad G_{i}=\sum_{\x^{(m)}\in S_i} w_m.
\]
Without loss of generality, we assume that each set $S_i$ contains at least one data point $\x^{(m)}$ and that the weights $\omega_m$ are positive, consistent with the assumption that $\omega(S_j\cap\Omega)>0$.
Hence, the diagonal entries of $\Gv$ are strictly positive.
Setting $\Cv=\Gv^{1/2}$ and adding the constraint \cref{eq_mult_matrix} in~\cref{EDMD_opt_prob2}, we obtain the following constrained least-squares problem:
\begin{equation} \label{eq_constraint_mult}
	\kmult\coloneqq\argmin_{\substack{\Kv\in \mathbb{C}^{N\times N}\\\Kv\text{ satisfies }\cref{eq_mult_matrix}}}\left\|\Wv^{1/2}\Psiv_Y\Gv^{-1/2}-\Wv^{1/2}\Psiv_X \Kv\Gv^{-1/2}\right\|_\F^2.
\end{equation}

\begin{example}[Relaxation of the permutation recovery problem]
	Permutation matrices satisfy the constraints in \cref{eq_mult_matrix}, and \cref{eq_constraint_mult} is related to the problem of finding a permutation matrix that minimizes the distance to a given matrix, known as permutation recovery problems. This problem is notoriously challenging to solve due to being NP-hard~\cite{emiya2014compressed,ma2021optimal,pananjady2017linear,unnikrishnan2015unlabeled}. Hence, we may view \cref{eq_constraint_mult} as a relaxation, where we do not assume that $\kmult$ is a permutation matrix. This ensures that the optimization problem can be solved efficiently.
\end{example}

\renewcommand{\algorithmicrequire}{\textbf{Input:}}
\renewcommand{\algorithmicensure}{\textbf{Output:}}
\begin{algorithm}[t!]
	\caption{MultDMD algorithm.}\label{alg_constraint}
	\begin{algorithmic}[1]
		\Require Piecewise constant basis functions $\Psi_1,\ldots,\Psi_N$ with disjoint support, data points $\x^{(1)},\ldots,\x^{(M)}, \y^{(1)},\ldots,\y^{(M)}\in \Omega$, and quadrature weights $\omega_1,\ldots,\omega_M$.
		\Ensure Matrix $\kmult\in \R^{N\times N}$ solving \cref{eq_constraint_mult}.
		\State Initialize the $N\times N$ matrices $\kmult$ and $\omega$ to zero.
		\For{$m=1,\ldots,M$}
		\State Find the indices $i_m$, $j_m$ such that $\psi_{i_m}(\x^{(m)}) = \psi_{j_m}(\y^{(m)})=1$.
		\State $\omega_{i_m,j_m} \gets \omega_{i_m,j_m} + w_m$
		\EndFor
		\State Compute the diagonal entries of the Gram matrix $G_j\gets\sum_{k=1}^N \omega_{j,k}$.
		\For{$i=1,\ldots,N$}
		\State $j_0 \gets \argmin_{1\leq j\leq N} (G_i-2\omega_{i,j})/G_j$.
		\State $[\kmult]_{i,j_0} \gets 1$.
		\EndFor
	\end{algorithmic}
\end{algorithm}

We propose an efficient algorithm, summarized in~\cref{alg_constraint}, to solve \cref{eq_constraint_mult} by characterizing the manifold of matrices satisfying the constraint \cref{eq_mult_matrix}. First, let $\fv=\bm e_i$ and $\gv=\bm e_j$ where $1\leq i\leq j\leq N$, and $\Kv\in \C^{N\times N}$ satisfying \cref{eq_mult_matrix}. Then, we have
$$
	\Kv \bm 1_{i,j}=(\Kv \bm e_i)\circ(\Kv \bm e_j),
$$
where $\bm 1_{i,j}(k) = 1$ if $k=i=j$ and $0$ otherwise. Taking $i=j$ leads to the constraints
\[
	\begin{bmatrix}
		K_{1,i}^2 \\
		\vdots    \\
		K_{N,i}^2
	\end{bmatrix}
	=
	\begin{bmatrix}
		K_{1,i} \\
		\vdots  \\
		K_{N,i}
	\end{bmatrix}, \quad 1\leq i\leq N,
\]
i.e., $K_{i,j}\in \{0,1\}$ for $1\leq i,j\leq N$.
While for any $i\neq j$, we have
\[
	\begin{bmatrix}
		K_{1,i}K_{1,j} \\
		\vdots         \\
		K_{N,i}K_{N,j}
	\end{bmatrix}
	=
	\begin{bmatrix}
		0      \\
		\vdots \\
		0
	\end{bmatrix}.
\]
Therefore, each row of $\Kv$ contains at most one non-zero element equal to one. For example, the following three matrices satisfy the constraint:
$$
	\begin{bmatrix}
		1 & 0 & 0 & 0 \\
		1 & 0 & 0 & 0 \\
		1 & 0 & 0 & 0 \\
		1 & 0 & 0 & 0 \\
	\end{bmatrix},
	\quad
	\begin{bmatrix}
		0 & 0 & 0 & 0 \\
		0 & 0 & 1 & 0 \\
		0 & 0 & 1 & 0 \\
		1 & 0 & 0 & 0 \\
	\end{bmatrix},
	\quad
	\begin{bmatrix}
		1 & 0 & 0 & 0 \\
		0 & 0 & 1 & 0 \\
		0 & 0 & 0 & 1 \\
		0 & 1 & 0 & 0 \\
	\end{bmatrix}.
$$
The set of matrices satisfying the constraint in \cref{eq_mult_matrix} consists of those with entries of either $0$ or $1$, each row containing at most one non-zero entry. Hence, \eqref{eq_constraint_mult} is equivalent to
\begin{equation} \label{eq_constraint_row_kdmd}
	\kmult=\argmin_{\substack{\Kv\in \{0,1\}^{N\times N}\\\text{rows have at most one 1}}} \sum_{m=1}^M\omega_m\sum_{j=1}^N\frac{1}{G_j}\left[\psi_j(\y^{(m)})-\Psi(\x^{(m)})K_j\right]^2,
\end{equation}
where $K_j$ denotes the $j$th column of $\Kv$. Let $1\leq m\leq M$, since the basis $\Psi$ is piecewise constant with disjoint support, there exists a unique set $S_{j_m}$, with $\y^{(m)}\in S_{j_m}$. Hence,
\[
	\psi_{j}(\y^{(m)})=\delta_{j,j_m}, \quad 1\leq j\leq N,
\]
where $\delta_{j,j_m}=1$ if $j=j_m$ and $0$ otherwise. Similarly, there exists $1\leq i_m\leq N$ such that $\psi_{i}(\x^{(m)})=\delta_{i,i_m}$. It follows that \cref{eq_constraint_row_kdmd} is equivalent to
\begin{equation} \label{eq_min_prob_B}
	\kmult=\argmin_{\substack{\Kv\in \{0,1\}^{N\times N}\\\text{rows have at most one 1}}} \sum_{m=1}^M\omega_m\left(\frac{1}{G_{j_m}}[1-K_{i_m,j_m}]^2+\sum_{j\neq j_m}\frac{1}{G_{j}}K_{i_m,j}^2\right).
\end{equation}
We introduce the matrix $\omega\in \R^{N\times N}$ defined as
\begin{equation} \label{eq_omega_ij}
	\omega_{i,j}=\sum_{\substack{m=1\\i_m=i,j_m=j}}^M\omega_m \quad \text{and}\quad G_j=\sum_{k=1}^N\omega_{j,k}
\end{equation}
so that \cref{eq_min_prob_B} is equivalent to
$$
	\kmult=\argmin_{\substack{\Kv\in \{0,1\}^{N\times N}\\\text{rows have at most one 1}}}\sum_{i=1}^N\sum_{j=1}^N \omega_{i,j}\left(\frac{1}{G_{j}}[1-K_{i,j}]^2+\sum_{\tilde{j}\neq j}\frac{1}{G_{\tilde{j}}}K_{i,\tilde{j}}^2\right).
$$
This problem decouples along each row, and for each $1\leq i\leq N$, we have the problem
$$
	\min_{\substack{K_{i,:}\in \{0,1\}^{N}\\\text{at most one 1}}} \sum_{j=1}^N \omega_{i,j}\left(\frac{1}{G_{j}}[1-K_{i,j}]^2+\sum_{\tilde{j}\neq j}\frac{1}{G_{\tilde{j}}}K_{i,\tilde{j}}^2\right).
$$
Suppose that there exists $j_0$ with $K_{i,j_0}=1$. Then, the above becomes
$$
	\min_{1\leq j_0\leq N}\left\{ \sum_{j\neq j_0} \omega_{i,j}\left(\frac{1}{G_{j}}+\frac{1}{G_{j_0}}\right)=\frac{1}{G_{j_0}}\left(-\omega_{i,j_0}+\sum_{j\neq j_0} \omega_{i,j}\right)+\sum_{j=1}^N \omega_{i,j}\frac{1}{G_j}\right\}.
$$
Since $\sum_{j=1}^N \omega_{i,j}\frac{1}{G_j}$ is independent of $j_0$, we arrive at
$$
	\min_{1\leq j_0\leq N}\frac{1}{G_{j_0}}\left(G_i-2\omega_{i,j_0}\right),\quad\text{with }G_i = \sum_{j=1}^N\omega_{i,j}.
$$
This problem requires $\mathcal{O}(M+N^2)$ operations to compute the optimal $\Kv$, which is the same complexity as standard EDMD for our choice of basis because the Gram matrix $\Psiv_X^*\Wv\Psiv_X$ is diagonal. However, MultDMD is often computationally faster in practice (see~\cref{sec_examples_MultDMD}) since it returns a sparse matrix $\kmult$ with at most $N$ non-zero elements, while $\kdmd$ is dense and does not enforce the multiplicative property of $\K$.

\subsection{Choice of basis functions}

The selection of a dictionary prompts interesting questions about the placement of support elements for $\psi_1,\ldots,\psi_N$ and how this affects the accuracy of approximating the map $\Fv:\Omega\to\Omega$ in \cref{eq_dyn_system}. There exists extensive research on the optimal arrangement of nodes for both interpolation and the approximation of continuous functions via piecewise discontinuous functions~\cite{baines1994algorithms,barrow1978unicity,de1973good,de2006good,du2005anisotropic,pittman2000fitting,tourigny1997analysis}.

We perform a Vorono\"i tesselation~\cite{aurenhammer1991voronoi} into $N$ cells, $S_1, \ldots, S_N$ and select the basis functions as indicator functions of these sets. The centroids $\muv_1,\ldots,\muv_N$ of the sets, with respect to the Euclidean distance, are chosen to minimize the following objective function:
\begin{equation} \label{eq_center_points}
	\min_{S_1,\ldots,S_N}\sum_{i=1}^N\sum_{\x^{(m)}\in S_i}\|\x^{(m)}-\muv_i\|_2^2,\quad\muv_i = \frac{1}{|S_i|}\sum_{\x^{(m)}\in S_i}\x^{(m)}.
\end{equation}
The clustering is performed using the $k$-means algorithm~\cite{lloyd1982least}, but one could also use a deep neural network to find the centroids' locations and distance~\cite{xie2016unsupervised}.

\begin{figure}[htbp]
	\centering
	\begin{overpic}[width=0.6\textwidth]{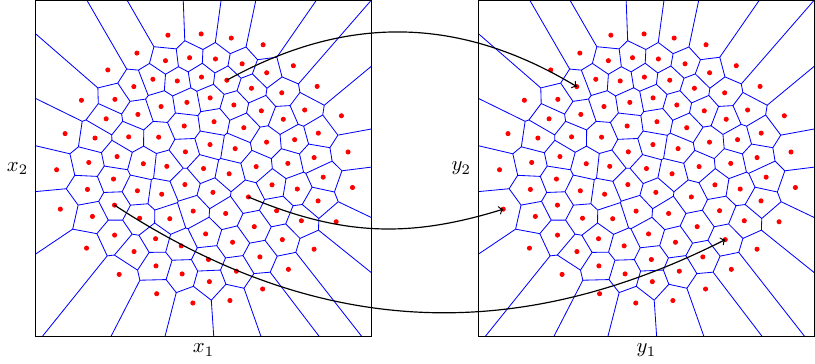}
	\end{overpic}
	\caption{Vorono{\"i} tessellation of the state space $\Omega$ computed by the $k$-means clustering algorithm using the data points $\x^{(1)},\ldots\x^{(m)}$. The basis functions are piecewise constant and locally supported on each Vorono\"i cell, whose centroid is highlighted by a red dot. The arrows represent the action of the Koopman operator.}
	\label{fig_voronoi}
\end{figure}

As an illustrative example, \cref{fig_voronoi} shows a Vorono{\"i} tesselation of the unit square and the resulting piecewise constant basis functions supported by the partition. The Koopman operator $\K$ is approximated by a matrix $\kmult$ mapping each basis function to at most one other basis function, which is represented by the arrows in \cref{fig_voronoi}. While \cref{eq_center_points} might be computationally expensive to solve for a large number of data points, one could compute centroids for a subsampled set of data points (e.g., by randomly or uniformly subsampling $\{\x^{(m)}\}$) and use the resulting partition to compute $\kmult$ for the full set of data points.

\paragraph{Dealing with high dimensions}
Identifying a suitable set of centroids and basis functions become very challenging in high spatial dimensions as clustering techniques usually suffer from the ``curse of dimensionality'', due to distances between points become less meaningful as the spatial dimension increases~\cite{zimek2012survey}. To address these, one could change the distance function to a more suitable one, or employ another clustering algorithm, such as mapping to a lower-dimensional space with principal component analysis~\cite{kriegel2009clustering}. In two of the examples below, we handle very high-dimensional systems by projecting onto POD modes. Another option is to project onto kernelized POD modes, similar to the approach used in kernelized EDMD~\cite{williams2015kernel}.

\section{Convergence analysis}\label{sec_convergence}

To motivate MultDMD theoretically, we provide a simple, sufficient condition for the convergence of DMD methods for unitary Koopman operators.

\subsection{Spectral measures}
Since $\K$ is unitary, the spectral theorem \cite[Thm.~X.4.11]{conway2019course} allows us to diagonalize $\mathcal{K}$ through a  \textit{projection-valued measure} $\mathcal{E}$. For each Borel measurable $S\subset\mathbb{T}$, $\mathcal{E}(S)$ is a projection onto the spectral elements of $\mathcal{K}$ inside $S$. Moreover,
\[
	g=\left(\int_\mathbb{T} 1\dd\mathcal{E}(\lambda)\right)g \quad\text{and}\quad \mathcal{K}g=\left(\int_\mathbb{T} \lambda\dd\mathcal{E}(\lambda)\right)g,\quad g\in L^2(\Omega,\omega).
\]
This formula decomposes $g$ based on the spectral content of $\K$. The projection-valued measure $\mathcal{E}$ simultaneously decomposes the space $L^2(\Omega,\omega)$ and diagonalizes the Koopman operator, akin to a custom Fourier-type transform that identifies coherent features. Scalar-valued spectral measures are of interest: for a normalized observable $g\in L^2(\Omega,\omega)$ with $\|g\| = 1$, the probability measure $\mu_g$ is defined as
$	\mu_g(S)=\langle \mathcal{E}(S)g,g \rangle.$
It is convenient to parametrize $\mathbb{T}$ using angle coordinates.
By changing variables to $\lambda=\exp(i\theta)$, measures $\xi_g$ on the interval $\pp$ are considered, with $\dd\mu_g(\lambda)=\dd\xi_g(\theta)$. This approach is applied to the projection-valued spectral measure, continuing to denote it as $\mathcal{E}$. The moments of the measure $\mu_g$ are the correlations
\begin{equation} \label{eq_correlations_moments}
	\langle \mathcal{K}^n g,g\rangle=\int_\mathbb{T} \lambda^n \dd\mu_g(\lambda)=\int_\pp e^{in\theta}\dd \xi_g(\theta),\quad n\in\mathbb{Z}.
\end{equation}
For example, if our system corresponds to dynamics on an attractor, these statistical properties enable the comparisons of complex dynamics~\cite{mezic2004comparison}. More generally, the spectral measure of $\mathcal{K}$ with respect to $g\in L^2(\Omega,\omega)$ is a signature for the forward-time dynamics of \cref{eq_dyn_system}.

\subsection{Weak convergence of spectral measures} \label{sec:weak_conv}
The convergence of matrix approximations of $\K$ is best analyzed through the convergence of spectral measures, which encode the spectral content of $\K$. Sometimes, approximations involve spaces outside of $\mathcal{H}=L^2(\Omega,\omega)$, like in DMD methods where approximations of the measure $\omega$ are derived from data. Consider a sequence of unitary operators $\{\mathcal{K}_n\}$ on finite-dimensional Hilbert spaces $\{\mathcal{H}_n\}$, where both $\mathcal{H}_n$ and $\mathcal{H}$ are closed subspaces of a larger Hilbert space $\mathcal{X}$. Denoting orthogonal projections as $\mathcal{P}_n:\mathcal{X}\rightarrow \mathcal{H}_n$ and $\mathcal{P}:\mathcal{X}\rightarrow \mathcal{H}$, we assume $\mathcal{P}_n^*\mathcal{P}_n$ strongly converges to $\mathcal{P}^*\mathcal{P}$ as $n\rightarrow\infty$. Given $g\in\mathcal{H}$ with $\|g\|=1$, we assume that $\mathcal{P}_ng\neq 0$ and let $\mathcal{E}_n$ be the projection-valued spectral measure of $\mathcal{K}_n$ and $\xi_{g,n}$ be the scalar-valued spectral measure of $\mathcal{K}_n$ with respect to $\mathcal{P}_ng/\|\mathcal{P}_ng\|$.

\subsubsection{Definitions}

We consider weak convergence of spectral measures defined as follows.

\begin{definition}[Weak convergence of measures]
	A sequence of Borel measures $\{\beta_n\}_{n\in\mathbb{N}}$ on $\pp$ \textit{converges weakly} to a Borel measure $\beta$ on $\pp$ if
	\begin{equation}\label{eq_meaning_weak}
		\lim_{n\rightarrow\infty}\int_\pp \phi(\varphi)\dd\beta_n(\varphi)= \int_\pp\phi(\varphi)\dd\beta(\varphi),\quad \forall\text{ continuous }\phi:\pp\rightarrow\mathbb{C}.
	\end{equation}
\end{definition}

The analog of weak convergence for projection-valued spectral measures is the functional calculus. If $F:\mathbb{T}\rightarrow\mathbb{C}$ is a bounded Borel function, then
\begin{equation}
	\label{eq_FC_def}
	F(\mathcal{K})=\int_{\pp} F(e^{i\varphi})\dd \mathcal{E}(\varphi)=\int_\mathbb{T}F(\lambda)\dd \mathcal{E}(\lambda).
\end{equation}

\begin{definition}[Functional calculus convergence]
	\label{def:FC_convergence}
	Given the above, $\mathcal{E}_n$ converges in the functional calculus sense to $\mathcal{E}$ if for any $g\in\mathcal{H}$ and continuous function $\phi:\pp\rightarrow\mathbb{C}$,
	\begin{equation}\label{eq_meaning_weak_proj_val}
		\lim_{n\rightarrow\infty}\underbrace{\int_\pp \phi(\varphi)\dd{\mathcal{E}}_n(\varphi)\mathcal{P}_ng}_{\phi(\log(\mathcal{K}_n)/i)\mathcal{P}_ng}= \underbrace{\int_\pp\phi(\varphi)\dd\mathcal{E}(\varphi)g}_{\phi(\log(\mathcal{K})/i)g}.
	\end{equation}
\end{definition}
The choice $\phi(\varphi)=\exp(im\varphi)$ yields convergence of $\mathcal{K}_n^m\mathcal{P}_ng$ to $\mathcal{K}^mg$ and hence the convergence of the Koopman mode decomposition for forecasting. The Portmanteau theorem \cite[Thm.~13.16]{klenke2013probability} yields the following corollaries.

\begin{corollary}
	\label{cor:spec_invs_weak}
	Suppose that $\xi_{g,n}$ converges weakly to $\xi_{g}$ for any $g\in\mathcal{H}$. If $U\subset\mathbb{T}$ is an open set with $U\cap\spec(\mathcal{K})\neq\emptyset$, then there exists $N\in\mathbb{N}$ such that if $n\geq N$ then $U\cap\spec(\mathcal{K}_n)\neq\emptyset$.
\end{corollary}

\begin{corollary}
	\label{cor:spec_poll_weak}
	Let $E\subset\pp$ be closed with $\exp(iE)\cap\spec(\mathcal{K})=\emptyset$, then
	\begin{enumerate}
		\item[(i)] The weak convergence of $\xi_{g,n}$ to $\xi_{g}$ implies $\lim_{n\rightarrow\infty}\xi_{g,n}(E)=0$;
		\item[(ii)] The convergence $\mathcal{E}_n$ in the functional calculus sense to $\mathcal{E}$ implies
		      \[
			      \lim_{n\rightarrow\infty} \int_E 1\dd{\mathcal{E}}_n(\varphi)\mathcal{P}_ng= 0,\quad \text{for all }g\in\mathcal{H}=L^2(\Omega,\omega).
		      \]
	\end{enumerate}
\end{corollary}

\cref{cor:spec_invs_weak} says that the discretization does not miss any parts of the spectrum, whereas \cref{cor:spec_poll_weak} ensures that the contribution of spurious eigenvalues diminishes as the measures converge. Weak convergence and convergence of the functional calculus are highly desirable and encompass important dynamical convergence and consistency.

\subsubsection{A simple sufficient condition}

The following theorem provides a sufficient condition for the convergence of spectral measures and the functional calculus under the condition that the approximation operators are unitary. Here, $T_n\xrightarrow[]{w}T$ means that $T_n$ converges weakly to $T$, i.e., $\lim_{n\rightarrow\infty}\langle T_n v,w\rangle=\langle T v,w\rangle$ for all $v,w\in \mathcal{X}$, while $T_n\xrightarrow[]{s}T$ means that $T_n$ converges strongly to $T$, i.e., $\lim_{n\rightarrow\infty} T_n v=T v$ for all $v\in \mathcal{X}$.

\begin{theorem}
	\label{thm_conv}
	Suppose that $\K$ and $\{\K_n\}$ are unitary and that $\mathcal{K}_n\mathcal{P}_n\xrightarrow[]{w}\mathcal{K}\mathcal{P}$. Then $\xi_{g,n}$ converges weakly to $\xi_g$ and $\mathcal{E}_n$ converges in the functional calculus sense to $\mathcal{E}$ as $n\rightarrow\infty$.
\end{theorem}

\begin{proof}
	Suppose that $\mathcal{K}_n\mathcal{P}_n\xrightarrow[]{w}\mathcal{K}\mathcal{P}$ and $g\in\mathcal{X}$. Then $\K_n\mathcal{P}_ng$ converge weakly in $\mathcal{X}$ to $\K\mathcal{P}g$ as $n\rightarrow\infty$. Since $\K_n$ is unitary and hence an isometry, $\|\K_n\mathcal{P}_ng\|=\|\mathcal{P}_ng\|$, which converges to $\|\mathcal{P}g\|=\|\K\mathcal{P}g\|$. Weak convergence of a sequence of vectors and convergence of norms implies convergence of the sequence of vectors. Hence, $\K_n\mathcal{P}_ng$ converges in $\mathcal{X}$ to $\K\mathcal{P}g$ and hence $\mathcal{K}_n\mathcal{P}_n\xrightarrow[]{s}\mathcal{K}\mathcal{P}$. Moreover, $\K_n^*\mathcal{P}_n\xrightarrow[]{w}\K^*\mathcal{P}$, since the adjoint operation is continuous with respect to the weak operator topology. The same argument now shows that $\mathcal{K}_n^*\mathcal{P}_n\xrightarrow[]{s}\mathcal{K}^*\mathcal{P}$.

	We prove that $\smash{\mathcal{E}_n}$ converges in the functional calculus sense to $\mathcal{E}$. The weak convergence of scalar-valued measures follows by taking an inner product and the fact that $\lim_{n\rightarrow\infty}\|\mathcal{P}_ng\|=\|\mathcal{P}g\|$. Trigonometric polynomials are dense in $\mathcal{C}_\pp$ and hence it is enough to consider $\phi(\varphi)=\exp(im\varphi)$ for $m\in\mathbb{Z}$ in \eqref{eq_meaning_weak_proj_val}. The functional calculus shows that for $g\in L^2(\Omega,\omega)$,
	$$
		\int_\pp  e^{im\varphi}\dd\mathcal{E}_n(\varphi)\mathcal{P}_nv=[\K_n \mathcal{P}_n]^mv,\quad
		\int_\pp  e^{im\varphi}\dd\mathcal{E}(\varphi) v=[\K \mathcal{P}]^mv.
	$$
	Since multiplication is jointly continuous in the strong operator topology and $\mathcal{K}_n\mathcal{P}_n\xrightarrow[]{s}\mathcal{K}\mathcal{P}$ and $\mathcal{K}_n^*\mathcal{P}_n\xrightarrow[]{s}\mathcal{K}^*\mathcal{P}$,
	$
		\lim_{n\rightarrow\infty}[A_n \mathcal{P}_n]^mv=[A \mathcal{P}]^mv
	$
	for all $m\in\mathbb{Z}$.
	The theorem follows.
\end{proof}

\subsection{Applications of \cref{thm_conv}}\label{sec:when_apply}
We can apply \cref{thm_conv} to mpEDMD~\cite{colbrook2023mpedmd} (\cref{sec_mpEDMD}) and periodic approximations~\cite{govindarajan2019approximation} (see \cref{sec_periodic}).

\begin{example}[mpEDMD]
	Suppose that $A_N=\mathcal{P}_N\mathcal{K}\mathcal{P}_N^*$, which corresponds to the large data limit of EDMD. Since $V_N$ is finite-dimensional, we can write
	$
		A_N=\mathcal{K}_N(A_N^*A_N)^{1/2},
	$
	for a unitary operator $\mathcal{K}_N$, where we take the non-negative square-root of the self-adjoint operator $A_N^*A_N$. The paper \cite{colbrook2023mpedmd} showed that mpEDMD is essentially a DMD-type approximation of the abstract operator $\mathcal{K}_N$. Moreover, it was shown that $\mathcal{K}_N\mathcal{P}_N\xrightarrow[]{s}\mathcal{K}$ and hence \cref{thm_conv}.
\end{example}

\begin{example}[periodic approximations]
	Periodic approximations explicitly enforce a bijection on a state-space partition to preserve the unitary property of the Koopman operator. The setup of \cite{govindarajan2019approximation} assumes that $\omega$ is finite and absolutely continuous with respect to the Lebesgue measure. They consider a sequence of partitions $P_n=\{p_{1,n},\ldots,p_{q_n,n}\}$ of $\Omega$ such that
	\begin{enumerate}
		\item[$(i)$] Each set in the partition is of equal measure:\footnote{This condition is why $\omega$ is restricted to have no atoms. In particular, the abstract theory surrounding the existence of periodic approximations assumes that the space $(\Omega,\omega)$ is a Lebesgue space. One should think of this as saying that $\omega$ is absolutely continuous with respect to Lebesgue measure up to atoms.} $\omega(p_{j,n})=\omega(\Omega)/q_n$;
		\item[$(ii)$] The maximum diameter converges to zero: $\lim_{n\rightarrow\infty}\max_{j=1,\ldots,q_n}\sup_{x,y\in p_{j,n}}|x-y|=0$;
		\item[$(iii)$] $P_{n+1}$ is a refinement of $P_n$.
	\end{enumerate}
	One can associate with the partition $P_n$ unitary operators $\K_n$ that act as a permutation on the sets in the partition.
	If $\Fv$ is continuous, then under the above conditions, there exists a choice of $\K_n$ with $\mathcal{K}_n\mathcal{P}_n\xrightarrow[]{s}\mathcal{K}$ \cite[Thm.~5.2]{govindarajan2019approximation}. This idea builds upon a rich tradition of studying ergodic systems through periodic approximations~\cite{halmos1944approximation,rokhlin1949selected,lax1971approximation,oxtoby1941measure}. For example, the Rokhlin--Halmos lemma \cite[Thm.~4.1.1]{sinai1989dynamical} says that periodic transformations (without necessarily enforcing the sets in the partition to have equal measure) are dense in the space of all automorphisms of a Lebesgue space equipped with a suitable topology. Various properties of the dynamical systems are connected to the rapidity of their approximation by the periodic ones~\cite{katok1967approximations}.
\end{example}

These examples serve as a theoretical motivation for our MultDMD algorithm introduced in \cref{sec_multiplicative}. Indeed, one can interpret MultDMD as a combination of mpEDMD with periodic approximations, along with suitable relaxations to make the resulting optimization problem feasible. Certain algorithmic choices we employ are motivated directly by \cref{thm_conv}. In particular, the proof of \cref{thm_conv} goes through for MultDMD with $n=N$ after the large data limit $M\rightarrow\infty$ (even in the case of non-unitary approximations) if:
\begin{itemize}
	\item MultDMD converges to $\K$ in the above weak sense;
	\item The norm of MultDMD applied to a given $g$ converges to $\|\K g\|$;
	\item The eigenvector matrix of MultDMD is uniformly conditioned as $n\rightarrow\infty$.
\end{itemize}
This final condition is needed to reduce the convergence when integrated against a general continuous $\phi$ to convergence when integrated against trigonometric polynomials. In practice, this can be achieved by clustering the data to ensure that the condition number of $\Gv$ remains bounded as $N \to \infty$, and by projecting onto the subspace of $V_N$ spanned by the eigenvectors of $\kmult$ associated with non-zero eigenvalues. The first two conditions can be satisfied under similar assumptions to those used in the periodic approximations discussed earlier. Finally, methods such as ResDMD can be employed to bound the error in the integral of the spectral measures of MultDMD against trigonometric polynomials, providing \textit{a posteriori} error bounds for the spectral measures.

\section{Numerical examples} \label{sec_examples_MultDMD}

We now consider four distinct examples. The first two are the nonlinear pendulum and the Lorenz system. These systems have Koopman operators with continuous spectra on $\mathbb{T}\backslash\{1\}$. The final two systems are standard benchmarks in fluid dynamics with noisy data and have a pure point spectrum. Note that the method of periodic approximations cannot be applied to any of these four systems.

\subsection{Nonlinear pendulum}
We consider the nonlinear pendulum given by
\begin{equation} \label{eq_nonlinear_pendulum}
	\dot{x}_1=x_2,\quad \dot{x}_2=-\sin(3x_1),\quad \x=(x_1,x_2)\in\Omega=[-\pi/3,\pi/3]_{\text{per}}\times \R.
\end{equation}
We simulate $400$ trajectories of the dynamical system \cref{eq_nonlinear_pendulum} by sampling the initial conditions on a uniform grid in $[-0.6,0.6]^2$ to ensure that the trajectories remain in the domain $[-1,1]^2$. We integrate the system in time up to $t=10$ using MATLAB's \texttt{ode45} function. The solutions are then interpolated at uniform times between $[0,10]$ with a time step $\Delta t = 0.1$, resulting in $M = 4\times 10^4$ snapshots.

\begin{figure}[htbp]
	\centering
	\begin{overpic}[width=\textwidth]{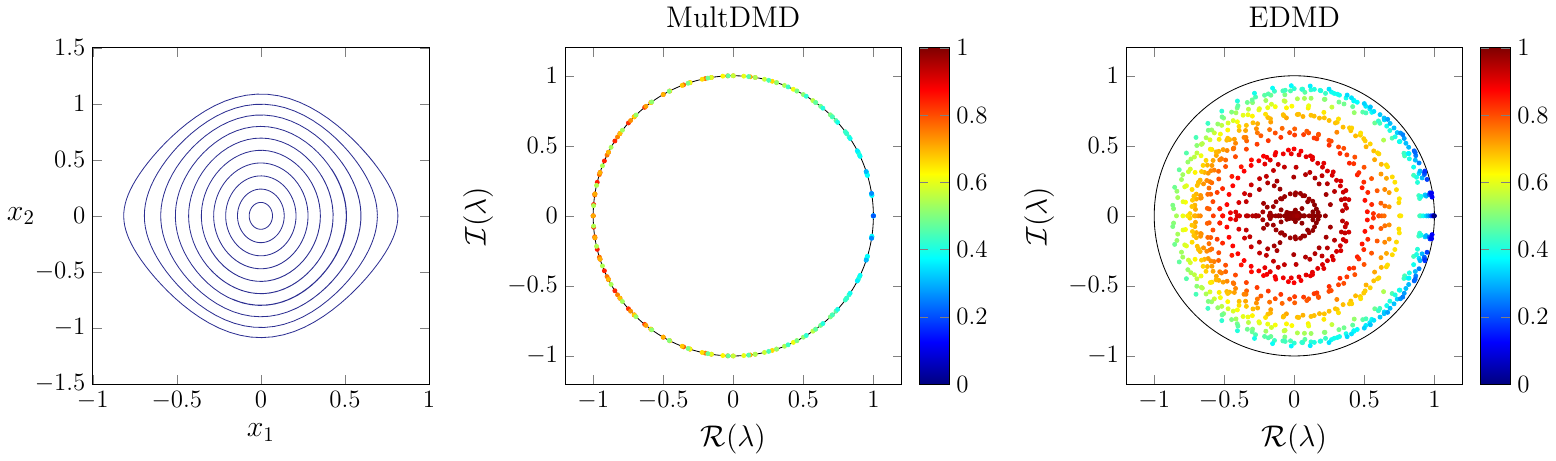}
	\end{overpic}
	\caption{Left: Trajectories of the nonlinear pendulum system~\cref{eq_nonlinear_pendulum}. Middle and right: Eigenvalues of the Koopman operator $\K$ for the nonlinear pendulum system computed using MultDMD and EDMD. The eigenvalues are colored according to the residual of the eigenpairs, estimated by the ResDMD method. The eigenvalues corresponding to eigenvectors supported on fewer than $50$ support elements have been discarded to take into account discretization error, as well as the zero eigenvalues computed by MultDMD. The unit circle is highlighted in black and indicates the spectrum of the Koopman operator.}
	\label{fig_eigenvalue_pendulum}
\end{figure}

The left panel in \cref{fig_eigenvalue_pendulum} displays examples of trajectories of the system. We select $N=1000$ basis functions by performing a Vorono\"i tesselation of the state space with the $k$-means algorithm using the data points $\x^{(1)},\ldots,\x^{(M)}$. The partition of the state space can be observed in \cref{fig_eigenvector_pendulum}. The initial centroid locations are chosen by selecting $N$ observations from $\x$ using the k-means++ algorithm~\cite{arthur2007k}, following MATLAB's default implementation of k-means. The quadrature weights $\{w_m\}_{m=1}^M$ used to discretize the integral in \cref{eq_ContinuousLeastSquaresProblem} are chosen to be uniform as $w_m = 1/M$ for simplicity. We then compute the matrix $\kmult$, which approximates the action of the Koopman operator on the finite-dimensional subspace spanned by the piecewise constant basis functions using MultDMD (\cref{alg_constraint}). We also compute the matrix $\kdmd$ using the standard EDMD procedure described in \cref{sec_edmd} as a comparison.

\begin{figure}[htbp]
	\centering
	\begin{overpic}[width=0.9\textwidth]{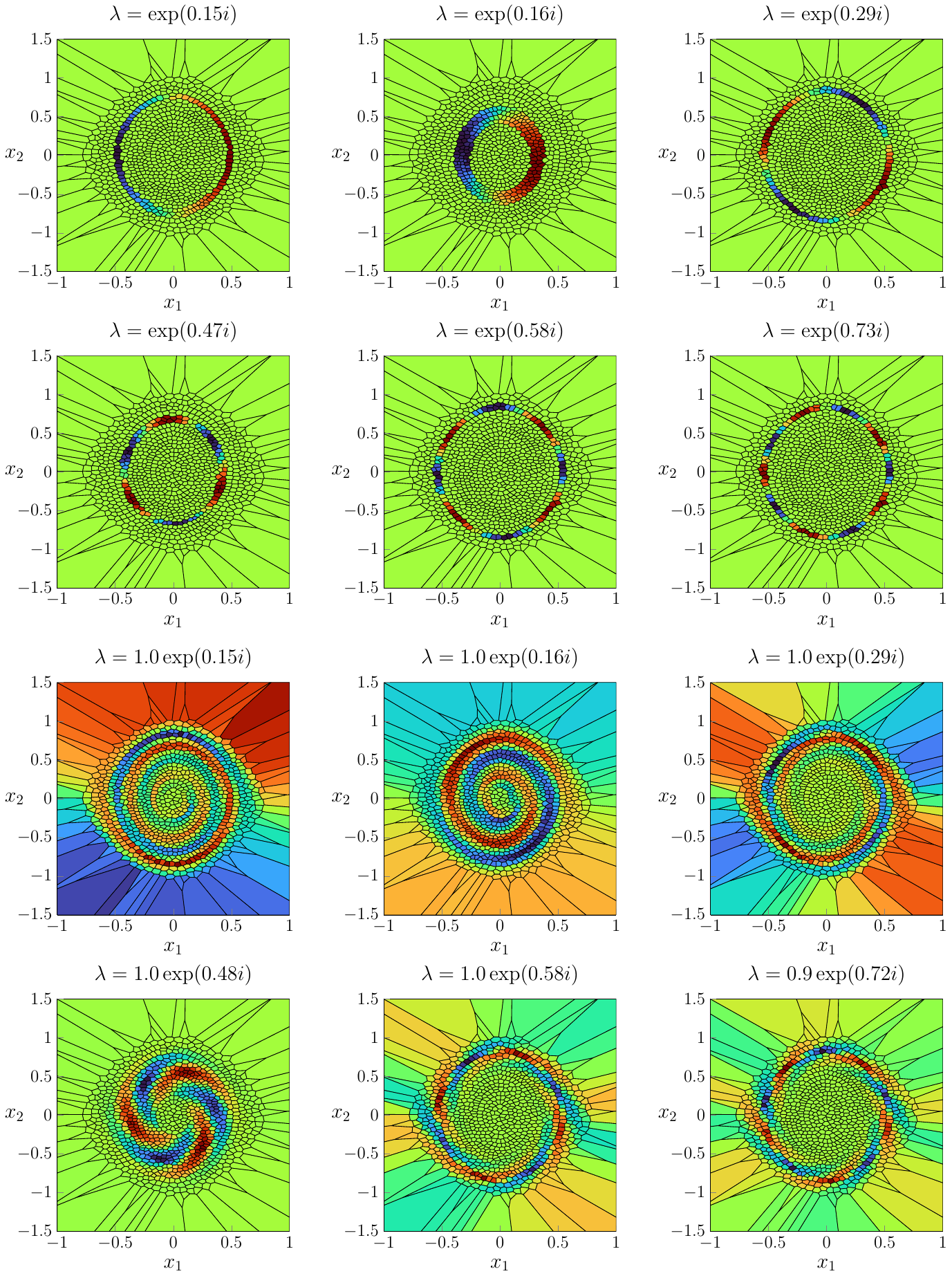}
		\put(-2,98){(a)}
		\put(-2,47.5){(b)}
	\end{overpic}
	\caption{Eigenfunctions for the pendulum system along with the corresponding eigenvalues calculated using the MultDMD algorithm (a) and EDMD algorithm (b). The eigenfunctions are normalized to have norm one.}
	\label{fig_eigenvector_pendulum}
\end{figure}

The middle and right panels of \cref{fig_eigenvalue_pendulum} display the eigenvalues of $\kmult$ and $\kdmd$, which are colored according to the residual of the eigenpair $(\lambda,\gv)$, computed by ResDMD~\cite{colbrook2023residualJFM,colbrook2021rigorous}. Namely, for a candidate eigenpair $(\lambda,\gv)$, the relative residual is defined as~\cite[Eq.~(3.2)]{colbrook2023residualJFM}
\begin{equation} \label{eq_res_dmd}
	\textrm{res}(\lambda,\gv) = \sqrt{ \frac{\gv^*[\Psiv_Y^*\Wv\Psiv_Y-\lambda[\Psiv_X^*\Wv\Psiv_Y]^*-\overline{\lambda}\Psiv_X^*\Wv\Psiv_Y+|\lambda|^2\Psiv_X^*\Wv\Psiv_X]\gv}{\gv^*[\Psiv_X^*\Wv\Psiv_Y]\gv}}.
\end{equation}
This is a convergent approximation of the (infinite-dimensional) residual $\|(\K-\lambda I)g\|/\|g\|$ of the Koopman operator $\mathcal{K}$. An advantage of our choice of basis functions is that \cref{eq_res_dmd} can be computed efficiently since the matrices $\Psiv_X^* \Wv\Psiv_X$ and $\Psiv_Y^* \Wv\Psiv_Y$ are diagonal, while $\Psiv_X^*\Wv\Psiv_Y$ is often sparse. If $\omega$ denotes the matrix defined in \cref{eq_omega_ij}, then
$$
	\Psiv_X^* \Wv\Psiv_X = \diag(\text{sum}(\omega,2)),\quad \Psiv_Y^* \Wv\Psiv_Y = \diag(\text{sum}(\omega,1)),
$$
where we employed MATLAB's notation for the sum of a matrix over its columns or rows. The eigenvalues of MultDMD lie on the unit circle, which is the spectrum of the Koopman operator. In contrast, many EDMD eigenvalues are spurious and inside the circle despite using the same dictionary as MultDMD.

Next, in \cref{fig_eigenvector_pendulum}, we display a selection of eigenfunctions computed using the MultDMD and EDMD algorithms. The EDMD eigenfunctions show severe dissipation, evidenced by the lack of coherency along constant energy surfaces. In contrast, the MultDMD eigenfunctions show modal structure along constant energy surfaces. Although the Koopman operators associated with the nonlinear pendulum have no eigenvalues (except $\lambda=1$), MultDMD reflects the generalized eigenfunctions of the Koopman operator, which are Dirac delta distributions of plane waves supported on constant energy surfaces~\cite{mezic2020spectrum, colbrook2024rigged}.

\begin{figure}[htbp]
	\centering
	\begin{overpic}[width=0.8\textwidth]{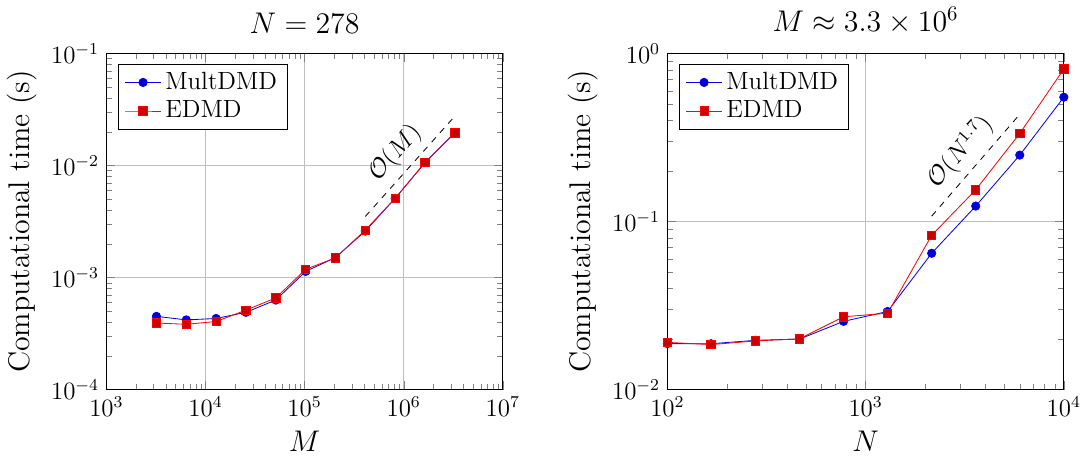}
	\end{overpic}
	\caption{Computational time for computing the Koopman matrix with EDMD and MultDMD as a function of the number of sample points $M$ at $N=278$ (left) and the number of basis functions $N$ at $M=2^{15}\times 100 \approx 3.3\times 10^6$ (right).}
	\label{fig_timing_dmd}
\end{figure}

Finally, in \cref{fig_timing_dmd}, we report the computational time\footnote{Timings were measured in MATLAB R2023a on a laptop with an Intel® Core™ i5-8350U CPU.} of EDMD and MultDMD with respect to the number of snapshots $M$ and the number of basis functions $N$. Here, we record the computational time needed for computing the approximation matrix to the Koopman operator given the set of basis functions (similar in both methods), \emph{i.e.} the timings of \cref{alg_constraint} in the case of MultDMD. While both methods have similar asymptotic behavior, we observe that MultDMD is slightly faster than EDMD, especially for large values of $N$. This is because the matrix $\kmult$ is sparse, while $\kdmd$ is dense.

\subsection{Lorenz attractor}
Next, we consider the Lorenz system~\cite{lorenz1963deterministic} described by three coupled ordinary differential equations:
\begin{equation} \label{eq_lorenz}
	\dot{x}=10\left(y-x\right),\quad\dot{y}=x\left(28-z\right)-y,\quad \dot{z}=xy-8z/3.
\end{equation}
We use the initial condition $x_0=y_0=z_0=1$ and consider the dynamics of $\x=(x,y,z)$ on the Lorenz attractor $\Omega$. The system is chaotic and strongly mixing~\cite{luzzatto2005lorenz} (and hence ergodic). It follows that the only eigenvalue (including multiplicities) of $\mathcal{K}$ is the trivial eigenvalue $\lambda=1$, corresponding to a constant eigenfunction. We collect snapshots of the system by integrating \cref{eq_lorenz} in time using MATLAB's \texttt{ode45} implementation of an explicit $(4,5)$ Runge--Kutta method up to a final time of $T=10^4$ with a time step of $\Delta t = 0.01$, resulting in $10^6$ snapshots. Then, we remove the first $10^4$ snapshots to neglect the effect of the initial condition and ensure that the trajectories lie on the Lorenz attractor, resulting in $M=9.9\times 10^5$ data points. Although this system is chaotic, we still have convergence as in \cref{quad_convergence} owing to the phenomenon of shadowing. Similarly to the nonlinear pendulum example, we perform a Vorono{\"i} tesselation of the state space with the $k$-means algorithm and select $N=5000$ basis functions. Given the large number of data points, we perform $k$-means on a smaller matrix consisting of $\x$ subsampled every ten time steps to find the centroid locations. The quadrature weights $\{w_m\}_{m=1}^M$ used to discretize the integral in \cref{eq_ContinuousLeastSquaresProblem} are chosen to be uniform as $w_m = 1/M$, which corresponds to ergodic sampling.

\begin{figure}[htbp]
	\centering
	\begin{overpic}[width=\textwidth]{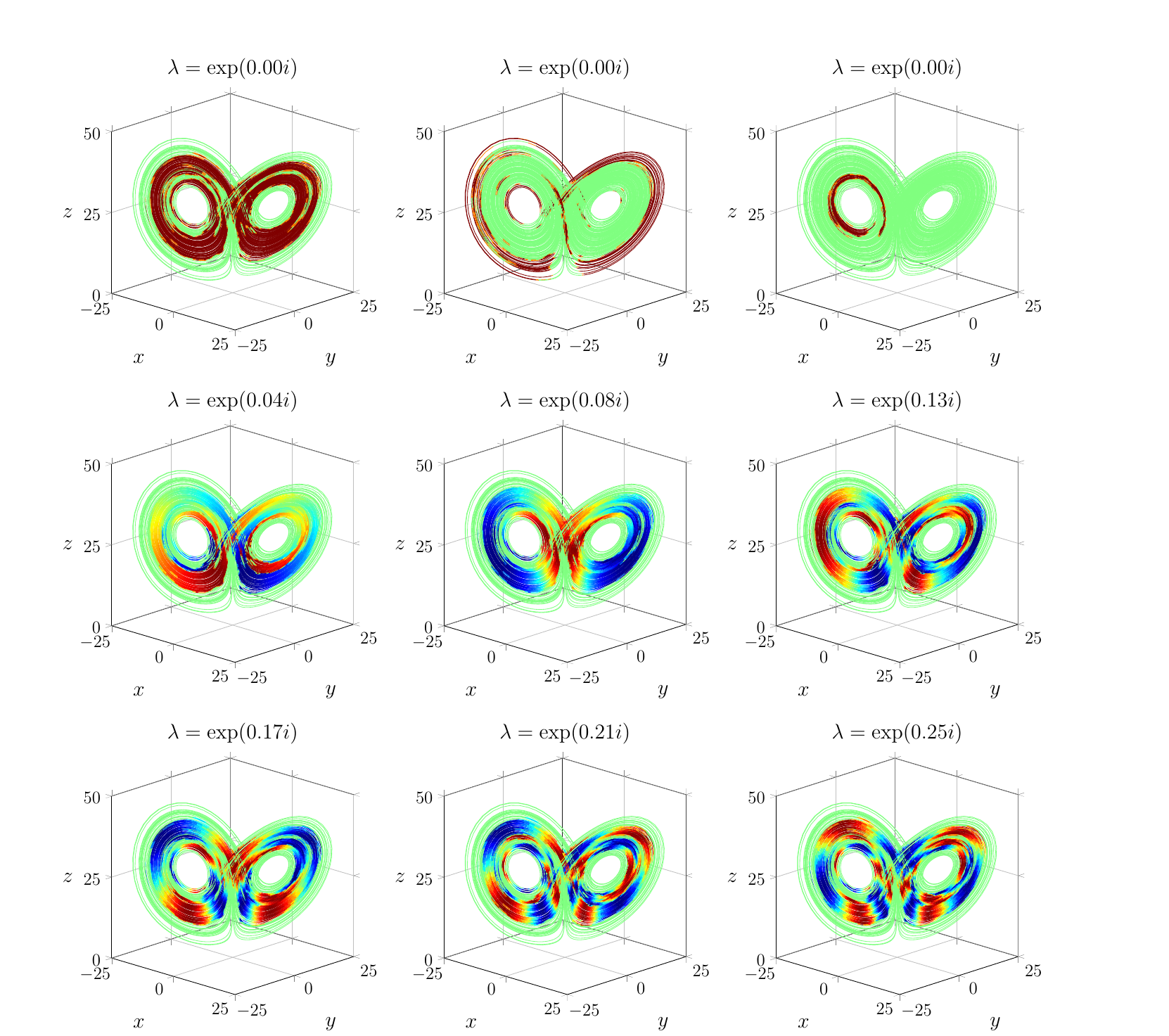}
	\end{overpic}
	\caption{MultDMD eigenfunctions for the Lorenz system and associated eigenvalues.}
	\label{figure_lorenz_eigenvector}
\end{figure}

\cref{figure_lorenz_eigenvector} displays a subset of the eigenfunctions computed using MultDMD. Although the system has no non-trivial eigenfunctions, the MultDMD eigenfunctions should be interpreted as approximate eigenfunctions and feature coherent structures. Moreover, the eigenvalues obtained by MultDMD satisfy a multiplicative group structure of the form $\{\lambda_0^k\}_k$, where $\lambda_0\approx e^{0.04 i}$, that approximate the continuous spectrum of the Lorenz system located on the unit circle. A movie available as Supplementary Material shows the evolution of the MultDMD eigenfunctions over time.

\begin{figure}[htbp]
	\centering
	\begin{overpic}[width=0.4\textwidth]{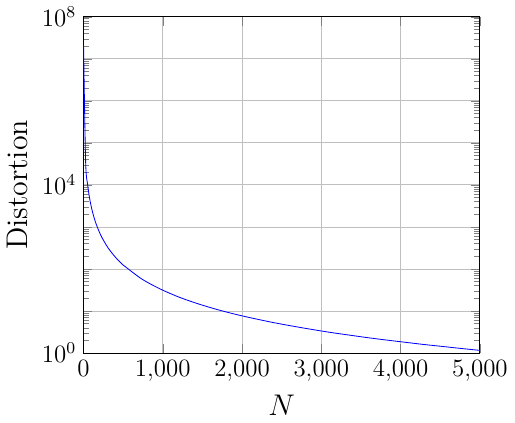}
	\end{overpic}
	\caption{Distortion (average distance to nearest centroid) of the $k$-means clustering algorithm for the Lorenz system as a function of the number of basis functions $N$.}
	\label{figure_elbow_lorenz}
\end{figure}

Finally, in \cref{figure_elbow_lorenz}, we display the distortion of the $k$-means clustering algorithm as a function of the number of basis functions $N$ for the Lorenz system example. The distortion is defined as the average distance of each data point to its nearest centroid and may be used to determine the appropriate number of basis functions. The plot shows an elbow point at $N=500$, indicating that one can reduce the number of basis functions without significantly affecting the quality of the approximation.

\begin{figure}[htbp]
	\centering
	\begin{overpic}[width=\textwidth]{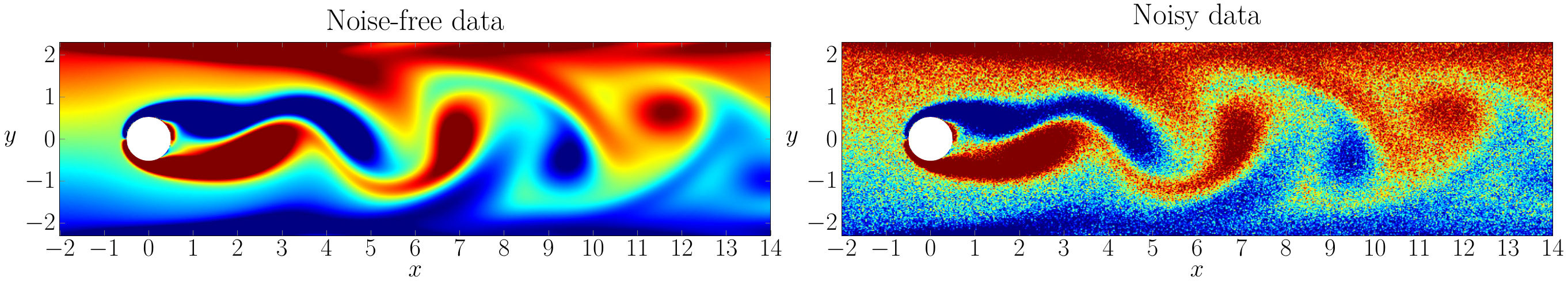}
	\end{overpic}
	\caption{Left: Snapshot of the vorticity for the cylinder wake example. Right: Snapshot of the vorticity with 40\% added Gaussian random noise.}
	\label{figure_noisy_cylinder_1}
\end{figure}

\subsection{Noisy cylinder wake}

For our next example, we consider a classical cylinder wake at $\textrm{Re}=100$. We use the dataset described in~\cite{colbrook2024another} (computed using an incompressible, two-dimensional lattice-Boltzmann solver \cite{jozsa2016validation,szHoke2017performance} at 158624 grid points) and consisting in vorticity field. We consider just 80 snapshots and corrupt the data with 40\% Gaussian random noise. A snapshot of the vorticity data, with and without noise, is shown in \cref{figure_noisy_cylinder_1}. To apply the MultDMD algorithm, we first project onto the first three POD modes. We then run MultDMD in this compressed state space with $N=80$. This example demonstrates how MultDMD can be applied to systems with an initially high-dimensional state space and to data corrupted with severe noise.

\begin{figure}[htbp]
	\centering
	\begin{overpic}[width=0.6\textwidth]{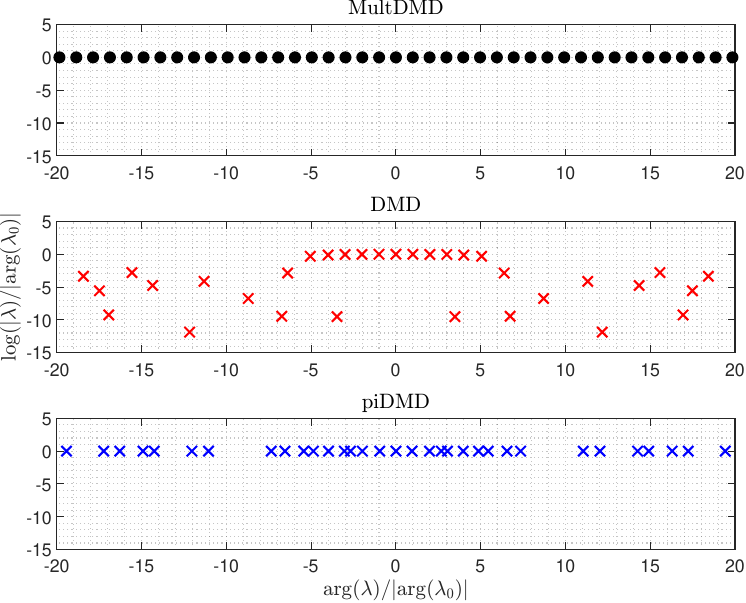}
	\end{overpic}
	\caption{Comparison of eigenvalues computed using MultDMD (top), exact DMD (middle), and mpEDMD/piDMD (bottom), which produce the same eigenvalues for this particular example when using a POD basis. The eigenvalues are normalized with respect to the base eigenvalue of $\lambda_0$ of the system. Hence, the exact eigenvalues correspond to $0,\pm1,\pm2,\ldots$ in the figure.}
	\label{figure_noisy_cylinder_2}
\end{figure}

At this Reynolds number, the flow is periodic. The Koopman operator associated with this system has a pure point spectrum consisting of powers of a base eigenvalue $\lambda_0$. \cref{figure_noisy_cylinder_2} shows the eigenvalues computed by MultDMD, exact DMD, and mpEDMD/piDMD.\footnote{piDMD and mpDMD produce the same eigenvalues for this particular example when using a POD basis.} While all of the standard methods are successful in the noise-free regime, they fail to approximate the spectrum accurately under severe noise perturbation. In contrast, MultDMD successfully captures the lattice structure of the spectrum in this high-noise regime. \cref{figure_noisy_cylinder_3} compares the dominant modes computed by MultDMD with the noisy data and the ground truth (noise-free). MultDMD successfully captures coherent features, even in the presence of severe noise perturbation.

\begin{figure}[htbp]
	\centering
	\begin{overpic}[width=\textwidth]{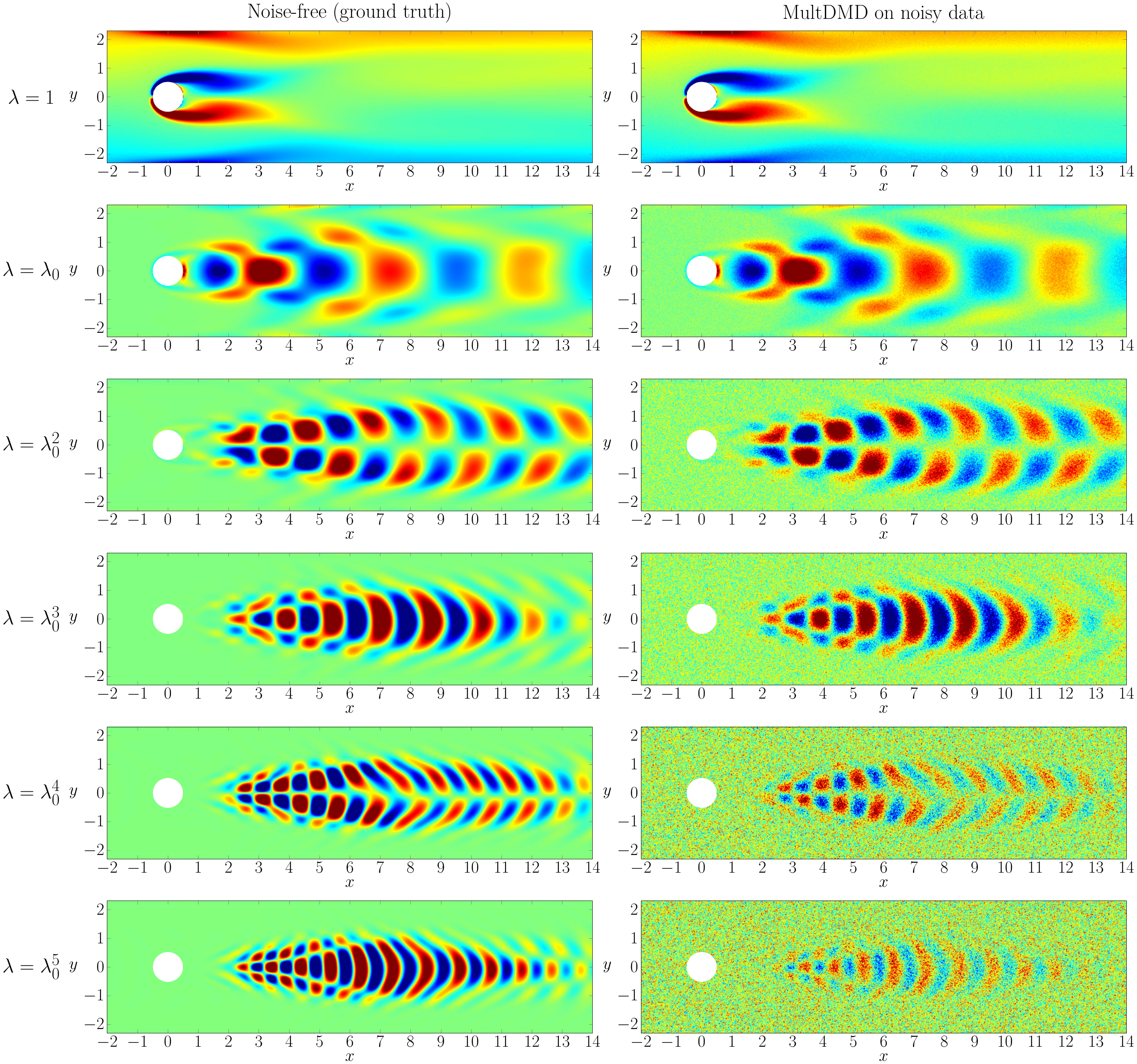}
	\end{overpic}
	\caption{The MultDMD modes of the cylinder wake corresponding to the mean and first five powers of $\lambda_0$ (negative powers of $\lambda_0$ can be obtained by conjugate symmetry).}
	\label{figure_noisy_cylinder_3}
\end{figure}

\subsection{Noisy lid-driven cavity}

As a final example, we consider the classical lid-driven cavity flow. The setup is an incompressible viscous fluid confined to a rectangular box with a moving lid. We use data from \cite{arbabi2017study} at $\textrm{Re}=16000$ computed using a Chebyshev collocation method with $65\times65=4225$ collocation points. We consider 1000 snapshots and corrupt the data with 40\% Gaussian random noise. A snapshot of the vorticity data, with and without noise, is shown in \cref{figure_noisy_cavity_1}. To apply the MultDMD algorithm, we first project onto the first five POD modes. We then run MultDMD in this compressed state space with $N=1000$.

\begin{figure}[htbp]
	\centering
	\begin{overpic}[width=0.7\textwidth]{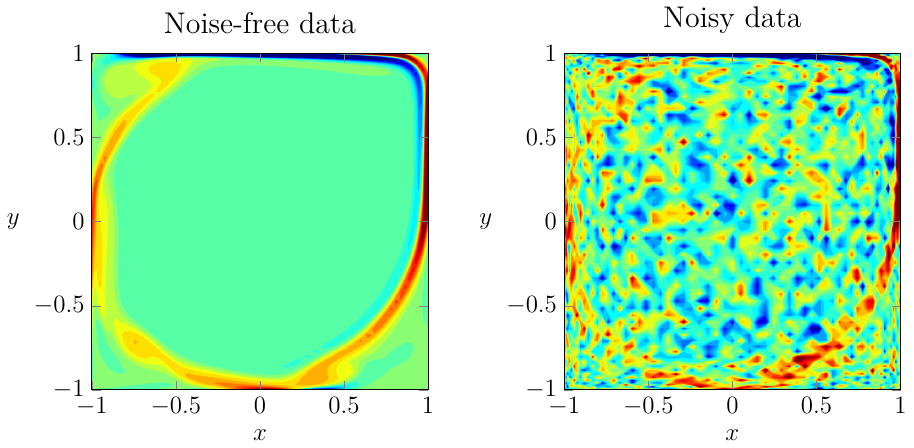}
	\end{overpic}
	\caption{Left: Snapshot of the vorticity for the lid-driven cavity. Right: Snapshot of the vorticity with 40\% added Gaussian random noise.}
	\label{figure_noisy_cavity_1}
\end{figure}

At this Reynolds number, the flow is quasiperiodic and posses an attractor diffeomorphic to a torus, with \textit{two} base Koopman eigenvalues $\lambda_1$ and $\lambda_2$ \cite[App.~B]{arbabi2017study}. The Koopman spectrum is pure point with eigenvalues $\lambda_1^n\lambda_2^m$ for $m,n\in\mathbb{Z}$. \cref{figure_noisy_cavity_2} shows the eigenvalues computed by MultDMD, exact DMD, and mpEDMD/piDMD.\footnote{Again, piDMD and mpDMD produce the same eigenvalues for this particular example when using a POD basis.} MultDMD is once again able to capture the lattice structure of the spectrum in this high noise regime, unlike all of the other methods. Finally, we compute the MultDMD modes corresponding to $\lambda=1,\lambda_1$ and $\lambda_2$.\footnote{These are not the eigenvalues closest to 1 in \cref{figure_noisy_cavity_2} due to the quasiperiodic structure of the spectrum.} These are shown in \cref{figure_noisy_cavity_3} and compared to the noise-free ground truth. We observe that MultDMD successfully captures coherent features, even in the presence of severe noise.

\begin{figure}[htbp]
	\centering
	\begin{overpic}[width=0.6\textwidth]{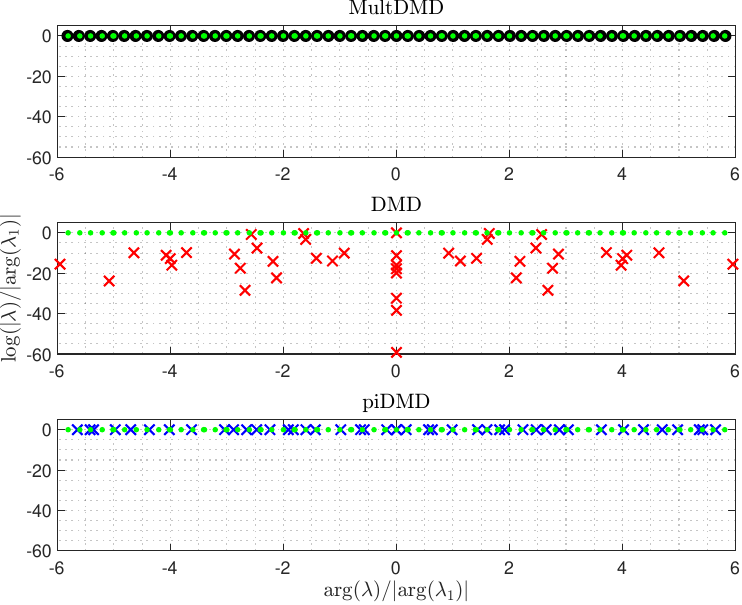}
	\end{overpic}
	\caption{Comparison of eigenvalues computed using MultDMD (top), exact DMD (middle), and mpEDMD/piDMD (bottom), which produce the same eigenvalues for this particular example when using a POD basis. The eigenvalues are normalized with respect to one of the base eigenvalues, $\lambda_1$, of the system. The dominant eigenvalues (computed with noise-free data and verified by comparing all three methods) are shown as green dots.}
	\label{figure_noisy_cavity_2}
\end{figure}

\begin{figure}[htbp]
	\centering
	\begin{overpic}[width=0.95\textwidth]{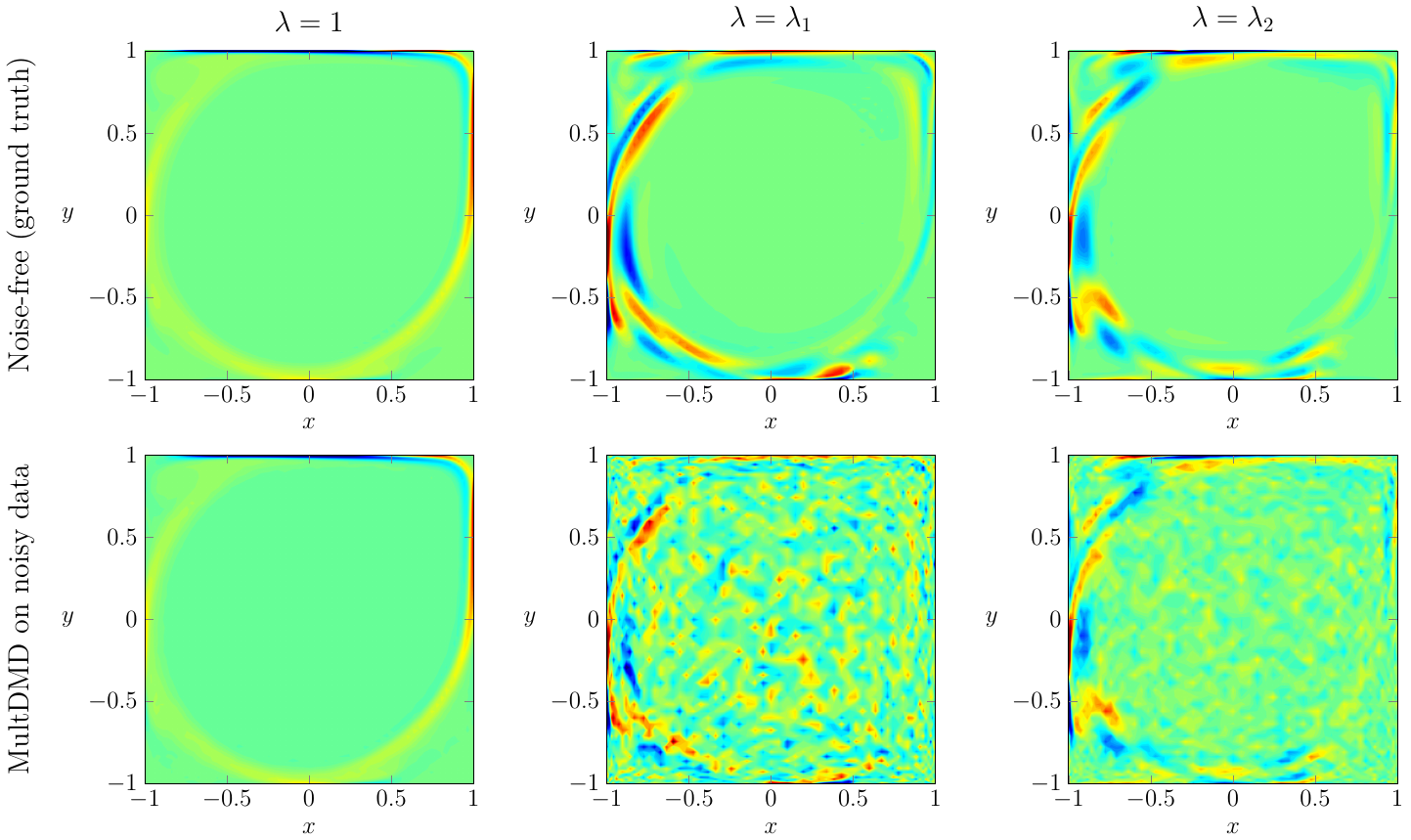}
	\end{overpic}
	\caption{The first three modes of the noise-free lid-driven cavity example (top row) along with the ones discovered by MultDMD on noisy data (bottom row).}
	\label{figure_noisy_cavity_3}
\end{figure}

\section{Conclusions}\label{sec_conclusions}

We introduced the Multiplicative Dynamic Mode Decomposition (MultDMD) as a robust approach for approximating the Koopman operator whilst preserving its inherent multiplicative structure within a finite-dimensional framework. The effectiveness of our method is evidenced through examples from various dynamical systems, including the nonlinear pendulum, the Lorenz system, and a noisy cylinder wake. Each example highlights MultDMD's capability to capture and preserve the spectral properties of the Koopman operator, proving to be especially proficient in environments afflicted by significant noise levels.

Numerous further extensions of this work could be considered. For example, the approximate point spectrum and the set of eigenvalues of general Koopman operators (not necessarily unitary) exhibit circular symmetry. It would be interesting to extend MultDMD to systems with a dissipative component or to apply it to the measure-preserving parts of systems. Another strategy could involve enforcing the multiplicative structure after forming a matrix approximation of the Koopman operator, possibly within a larger tensor-product space of observables. One could also consider adaptive methods that dynamically adjust our basis functions based on the evolving characteristics of the dynamical system. While MultDMD has shown promise in computational efficiency, there is scope for algorithmic improvements in very high-dimensional systems. Finally, deeper theoretical investigations into the convergence properties of MultDMD, especially concerning its ability to recover exact Koopman operators under varying conditions, would be beneficial. A challenge here is that MultDMD does not necessarily lead to unitary approximations.


\bibliographystyle{siamplain}
\bibliography{bib_file_DMD3}

\end{document}

%% file: ex_shared.tex
\usepackage[table]{xcolor}
\usepackage[scaled=.98]{BOONDOX-uprscr}
\usepackage[utf8]{inputenc}
\usepackage{yfonts}
\usepackage{eufrak}
\usepackage[english]{babel}
\usepackage{graphicx}
\usepackage{fancyhdr}
\usepackage{amsfonts,amssymb,amsmath,amssymb,mathtools,enumitem}
\usepackage[percent]{overpic}
\usepackage{tikz,mathrsfs,xargs}
\usepackage[colorinlistoftodos,prependcaption,textsize=tiny]{todonotes}
\newcommandx{\at}[2][1=]{\todo[linecolor=red,backgroundcolor=red!25,bordercolor=red,#1]{#2}}

\DeclareTextFontCommand{\textgoth}{\mathfrak}
\DeclareMathAlphabet{\mathpzc}{OT1}{pzc}{m}{it}

\usepackage{hyperref}
\usepackage{booktabs}
\usepackage{varwidth}
\usepackage{cleveref}
\usepackage{cite}
\usepackage{algorithm,float}
\usepackage{algpseudocode}

\definecolor{rev1}{HTML}{cb270f}
\definecolor{rev2}{HTML}{1c8235}
\definecolor{rev3}{HTML}{3815B9}

\DeclareMathOperator*{\argmin}{argmin}

\usepackage{ifthen,tabularx,graphicx,multirow}

\usepackage{rotating}
\usepackage{etoolbox}
\usepackage{tcolorbox}
\usepackage{bbm}

\newcolumntype{Y}{>{\centering\arraybackslash}X}
\definecolor{ddg}{rgb}{0.2, 0.2, 0.2}
\definecolor{llg}{rgb}{0.95, 0.95, 0.95}
\tikzstyle{stuff_fill}=[rectangle,draw,fill=pink,minimum size=0.5em]

\usepackage{enumitem}
\setlist[enumerate]{leftmargin=.5in}
\setlist[itemize]{leftmargin=.5in}

\newsiamremark{remark}{Remark}
\newsiamremark{example}{Example}
\newsiamremark{hypothesis}{Hypothesis}
\crefname{hypothesis}{Hypothesis}{Hypotheses}
\newsiamthm{claim}{Claim}

\headers{Multiplicative dynamic mode decomposition}{N. Boull\'e and M. J. Colbrook}

\title{Multiplicative Dynamic Mode Decomposition\thanks{This work was funded by an INI-Simons Research Fellowship and the Office of Naval Research (ONR) under grant N00014-23-1-2729 (NB), and a Cecil King Travel Scholarship from the Cecil King Foundation and the London Mathematical Society (MJC).}}

\author{Nicolas Boull\'e\thanks{Department of Mathematics, Imperial College London, London, SW7 2AZ, UK (\email{n.boulle@imperial.ac.uk}).}
  \and Matthew J. Colbrook\thanks{Department of Applied Mathematics and Theoretical Physics, University of Cambridge, Cambridge, CB3 0WA, UK (\email{m.colbrook@damtp.cam.ac.uk}).}}

